\documentclass[12pt]{article}

\usepackage{amsmath}
\usepackage{amssymb}
\usepackage{amsthm}
\usepackage{graphicx}
\usepackage{amscd}
\usepackage{epic, eepic}
\usepackage{url}
\usepackage{color}

\newcommand{\z}{\mathrm{z}}

\begin{document}
\title{Density version of the Ramsey problem and  the directed Ramsey problem}

\author{{\bf Zolt\'an L\'or\'ant Nagy}\\ 
{\small MTA--ELTE Geometric and Algebraic Combinatorics Research Group}\\
{\small H--1117 Budapest, P\'azm\'any P.\ s\'et\'any 1/C, Hungary}\\
{\small \tt{  nagyzoli@cs.elte.hu}}}

\date{}

 \newtheorem{theorem}{\bf Theorem}[section]
 \newtheorem{lemma}[theorem]{\bf Lemma}
 \newtheorem{cor}[theorem]{\bf Corollary}
  \newtheorem{notion}[theorem]{\bf Notion}
 \newtheorem{prop}[theorem]{\bf Proposition}
 \newtheorem{conj}[theorem]{\bf Conjecture}
\newtheorem{example}[theorem]{\bf Example}
\newtheorem{remark}[theorem]{\bf Remark}
\newtheorem{defn}[theorem]{\bf Definition}
\newtheorem{defi}[theorem]{\bf Definition}
\newtheorem{problem}[theorem]{\bf Problem}
\newtheorem{Claim}[theorem]{\bf Claim}
\newtheorem{construction}[theorem]{\bf Construction}
\newtheorem{obs}[theorem]{\bf Observation}
\setcounter{equation}{0}

\maketitle

\begin{abstract} 
We discuss a variant of the Ramsey and the directed Ramsey problem. First, consider a complete graph on $n$ vertices and a two-coloring of the edges such that every edge is colored with at least one color and the number of bicolored edges $|E_{RB}|$ is given. 
The aim is to find  the maximal size $f$ of a monochromatic clique which is guaranteed by such a coloring. Analogously, in the second problem we consider semicomplete digraph on $n$ vertices such that the number of bi-oriented edges $|E_{bi}|$ is given. The aim is to bound the size $F$ of the maximal transitive subtournament that is guaranteed by such a digraph.

Applying probabilistic and analytic tools and constructive methods we show that if $|E_{RB}|=|E_{bi}| = p{n\choose 2}$, ($p\in [0,1)$),  then $f, F < C_p\log(n)$ where $C_p$ only depend on $p$, while if $m={n \choose 2} - |E_{RB}| <n^{3/2}$ then $f= \Theta (\frac{n^2}{m+n})$. The latter case is strongly connected to Tur\'an-type extremal graph theory.



\end{abstract}

\section{Introduction} 

Ramsey's theorem  concerns  one of the classic questions in graph theory, aiming to determine accurate bounds on the Ramsey
number $R(k)$, the smallest number $n$ such that in any two-coloring of the edges of a  complete graph on
$n$ vertices, there is guaranteed to be a monochromatic clique of size $k$. An inverse approach to this problem is the following.  Let $c:   E(K_n)  \rightarrow \{\mbox{red, blue}\}$ be a $2$-coloring of the edges of the complete graph $K_n$. What is the largest number $f(n)$ such that there exists a monochromatic clique of size $f(n)$ in any  coloring $c$ of $K_n$? 

Clearly, $f(R(k))=k$, moreover $t<R(k)$ implies $f(t)<k$.
\smallskip

Due to Erd\H os, Szekeres and Spencer \cite{erd, esz, spen}, it is well known that the following bounds hold for $f(n)$ if $n\geq 2$:

\begin{equation}
 \frac{1}{2}\log_2(n) \leq f(n)\leq 2\log_2(n)\label{szeker}. \end{equation}
\noindent Although it is widely investigated, even the most significant improvements yet not have any affect on the main terms, of both bounds \cite{conl, spen}.

\medskip

In this paper we propose a variant of this problem.
Let us color the edges of $K_n$ in such a way that every edge is either unicolored (red, blue), or bicolored. That is, our modified general two-coloring function is of form $$c:   E(K_n)  \rightarrow \{\mbox{red, blue, red and blue}\}.$$ 
Let $E_{RB}(K_n)$ denotes the set of bicolored edges of $K_n$, and  $m= { n \choose 2}- |E_{RB}(K_n)| $ denotes the number of unicolored edges in a two-coloring of $K_n$. A subgraph $H$ of $K_n$ is called monochromatic in color red, respectively blue, if each edge of $H$ is colored with red, respectively blue color. (That is, this extension allows a subgraph to be monochromatic in both colors, if every edge is bicolored.)
 
\begin{problem}\label{elso}
Find the the size $f(n, m)$ of the largest monochromatic clique which is contained in any two-coloring of $K_n$ with $m$ unicolored edges. 
\end{problem}

\noindent Our goal is to study the order of magnitude of $f(n, m)$ in terms of $n$ and $m$, where $m$ is considered as a function of $n$.
\bigskip


Our second problem concerns the directed Ramsey number $\vec{R}(n)$ which was defined by Erd\H os and Moser, and it had been studied by various authors \cite{erdos, stearn, bang, reid, tabib, lara, flores}. 
Here the aim is to determine the size of the largest transitive subtournament  which appears in any tournament on $n$ vertices.
A \textit{tournament} $T$ is a digraph in which every pair of vertices is joined by exactly one directed edge. A \textit{subtournament} $T'$ of $T$ is a tournament induced by $V(T')\subseteq V(T)$. A tournament is \textit{transitive} if it is acyclic. Recall that this is equivalent to the property that it has a topological ordering, that  is, a linear ordering of its vertices such that for every directed edge $\vec{uv}$, $u$ comes before $v$ in the ordering.\\
The inverse approach appeared also in this problem, in the following sense.
 Let $F(n)$ denote the greatest integer $F$ such that all tournaments of order $n$ contain the transitive subtournament of order $F$.
 In \cite{stearn}, Stearns showed that $F(2n)\geq F(n)+1$ which provides $F(n)\geq \left\lfloor \log_2{n}\right\rfloor+1$. On the other hand, Erd\H os and Moser proved that $F(n)\leq \left\lfloor 2\log_2{n}\right\rfloor+1$, and conjectured that the lower bound of Stearns in fact holds with equality.  However this turned out to be false \cite{reid}. Later the lower bound was improved by  Sanchez-Flores \cite{flores}, who proved  $F(n)\geq \left\lfloor \log_2{(n/55)}\right\rfloor + 7$ by determining the function $F(n)$ for values $F\leq55$ using computer techniques,  and applying Stearns' recursion. Note that although this result improves the lower bound on $F(n)$, it does not disprove the asymptotic $F(n)=(1+o(1))\log_2(n)$ version of the conjecture.\\

A \textit{biorientation} of an undirected graph $G$ is a digraph obtained from $G$ by replacing
each edge $\{v_i, v_j\}$ of $G$ with either the arc $\overrightarrow{v_iv_j}$  or the arc $\overrightarrow{v_jv_i}$ or both of them.
A \textit{semicomplete digraph} is a biorientation of a complete graph $K_n$.  If both directed edges lie in the edge set, then we call the undirected edge $(v_i, v_j)$  bi-oriented.  Hence, both tournaments and complete digraphs are examples for this digraph family. 

Let $G$ be a semicomplete digraph on $n$ vertices. We call $G$ transitive, if it contains a transitive subtournament $T$ such that $V(T) = V(G)$. Let $E_{bi}(G)$ be the set of bi-oriented edges of the semicomplete digraph $G$, and  $m= { n \choose 2}- |E_{bi}(G)| $ denotes the number of  edges of one orientation in $G$.

Following the concept of the general two-coloring and Problem \ref{elso}, we study the following

\begin{problem}\label{masodik}
Find the size $F(n, m)$ of the largest transitive tournament which is contained in any semicomplete digraph  $G$ of $m$ edges of one orientation, on $n$ vertices. 
\end{problem} 

 The aim of the present paper is to study the order of magnitude of the considered functions  $f(n, m)$ and $F(n, m)$ in terms of the cardinality of the bi-oriented or bicolored edges which increases from $O(\log (n))$ to $n$ while $|E_{RB}(K_n)|$, respectively $|E_{bi}(G)|$ increase till ${n \choose 2}$. 
 
This problem has a connection to the problem of Erd\H os and Rado too \cite{ER}, studied also by Larson and Mitchell \cite{Lars}. The question they discuss is the following. Given $N>2$ and $k>1$, what is the smallest $n$ so that every digraph on a set of $n$ vertices either has an independent set of size $k$ or contains a transitive tournament $TT_N$ on $N$ vertices.
We also mention the concept of Ramsey-Tur\'an theory due to T. S\'os and M. Simonovits, in which the function $RT(n, L, \mu)$, the maximum number of edges of an $L$-free graph
on $n$ vertices with independence number less than $\mu$ plays a key role (see e.g. \cite{sos, balogh}).
In the center of  interest concerning this function, it is the asymptotic behavior of $RT(n, L, \phi(n))$ and its so called 'phase transition',
that is, when and how the asymptotic behavior of $RT(n, L, \phi)$ changes
sharply when we replace $\phi$ by a slightly smaller function $\phi'$.
 

\noindent Our first result concerning the relation of the two functions in view  is 
 
 \begin{prop} \label{relation}  $f(n, m) \leq F(n, m).$
 \end{prop}

\noindent The following observation is straightforward.
\begin{obs} \label {triv} 
$f(n, m={n \choose 2})=f(n)$, $f(n, m=0)=n$,\\
$F(n, m={n\choose 2})=F(n)$, $F(n, m=0)=n$,\\
$f(n,m)\leq f(n, m')$ and  $F(n,m)\leq F(n, m')$ if $m'\leq m$ (monotonicity).
\end{obs}

\noindent The main result of this paper describes some sort of phase transition as well. In the case $|E_{RB}(K_n)|$, respectively $|E_{bi}(G)|$ are small, and $m=\Omega(n^2)$, $f(n,m)$ and $F(n,m)$ are of order $\log{n}$, while $m=o(n^{3/2})$ implies that $f(n,m)$ is of order $\frac{n^2}{m+n}$.  Also, seemingly random like extremal structures give their place to well-organized, Tur\'an-type structures.
Formally this can be expressed as below.

\begin{notion} Let $m=\Theta(n^2)$, then introduce the parameter $p$ as the density of the $2$-colored or bi-oriented edges, $$p:= \frac{|E_{RB}(K_n)|}{|E(K_n)|}<1 \mbox{\ \ in Problem \ref{elso} and similarly \ \ } p:= \frac{|E_{bi}(G)|}{|E(K_n)|}<1  \mbox{\ \ in Problem \ref{masodik}. }$$
 We call a $2$-coloring, respectively a semicomplete digraph $G$ {\em $p$-dense}, if $p$ is the ratio of the bicolored edges of $K_n$, respectively the bi-oriented edges of $G$. Suppose that $p \in [0,1)$ is fixed. Then  $f_p(n)$ denotes the
largest monochromatic clique that is guaranteed in a $p$-dense coloring, that is, $f_p(n): = f(n, (1-p){n\choose 2} )$. $F_p(n)$ denotes 
the largest transitive tournament that is guaranteed in a $p$-dense digraph, that is, $F_p(n): = F(n, (1-p){n\choose 2} )$. 
\end{notion}


\begin{theorem}\label{fo1}
$$ F_p(n) \leq \frac{2}{1-p}\log (n(1-p)),$$
and
$$\frac{0.2}{1-p}\frac{1}{\log{\frac{2}{1-p}}}\log n \leq f_p(n) \mbox{ \ \ 
if \ \ } p\ge 2/3.$$
\end{theorem}

Note that $\frac{1}{2}\log n\leq f_p(n) $ holds for every $p$.

\noindent In the other case, let $m=o(n^2)$.

\begin{theorem}\label{fo2}

\item[(i)]\ If $m\leq n$, then $$f(n,m)= n- \left\lfloor\frac{m}{2}\right\rfloor \mbox{ \ \ and \ \ } F(n, m) = n- \left\lfloor\frac{m}{3}\right\rfloor.$$

\item[(ii)] If $m\geq n$, then  $$\mbox{(A) \ \ } f(n,m) \geq \frac{n^2}{m+n} \mbox{ \ \ and \ \ (B) \ \ } F(n, m)  \geq  \frac{2n^2}{2m+n}.$$

\item[(iii)] If $m\leq n^{3/2}$, then  $$f(n, m) \leq \frac{n}{\left\lfloor\frac{m}{n}\right\rfloor +1}.$$




\end{theorem}


Our paper is built up as follows. 
In Section 2, we begin with the proof of Proposition \ref{relation}. Next we prove Theorem  \ref{fo1} applying constructive and probabilistic methods, and a quantitative variant of the Erd\H os--Szemer\'edi theorem \cite{ESZ}. 


  In Section 3, we examine the case $m= o(n^2)$ and prove Theorem  \ref{fo2}. First we discuss the case $m\leq n$ when  exact results can be obtained. Then we prove the remaining statements of  Theorem  \ref{fo2} (ii) and (iii). The lower bounds are derived from the Caro-Wei bound \cite{Caro, wei} and the results of Alon, Kahn and Seymour on $k$-degenerate graphs \cite{AKS},  while the upper bounds are related to equitable colorings, graph packings, and the directed feedback vertex set problem.  Finally,  a slight improvement is presented on the upper bound of $F(n)$ using probabilistic techniques in Section 4.



\section{The case $m=\Theta(n^2)$ }


In this section, we explain the connection between our two problems in consideration, then  give upper bounds on the functions $f$ and $F$ when  $m=\Theta(n^2)$.


\begin{proof}[Proof of Proposition \ref{relation}]  
It is enough to prove that if  a bound $F(n, m)< k$ holds for a $k \in \mathbb{Z^+}$, then  $f(n,m)< k$ also holds. Consider a semicomplete digraph $G$ having exactly $m$ one-oriented edges, and assume that no transitive subtournament exists on $k$ vertices in  $G$. Label its vertices by $\{1,2, \ldots, n \}$ arbitrary, and assign red color to ascending edges $\overrightarrow{ij}$ (where $i<j$), blue color to descending edges $\overrightarrow{ij}$ (where $i>j$). The coloring we thus obtain cannot contain a monochromatic clique of size $k$. Indeed, that would provide a transitive subtournament in $G$.
\end{proof}

\subsection{Probabilistic argument, upper bound on $F_p(n)$}

Our aim is to generalize the theorem of Erd\H os and Moser \cite{erdos} to confirm the upper bound on $F_p(n)$, Theorem \ref{fo1}. To this end, we prove a lemma first. $\mathbb{E}(X)$ will denote the expected value of a random variable $X$.

\begin{lemma}\label{moment} Let $p\in [0,1]$ and  
let $Y$ be a random variable of distribution  $Y \sim Binom(n, p)$ and $Z$ be a random variable  of distribution $Z\sim Hypergeom(N, pN, n)$. Then for every $c\geq 1$, $$\mathbb{E}(c^Z)\leq  \mathbb{E}(c^Y).$$ 
\end{lemma}

\begin{proof}  $$\mathbb{E}(c^Z)= \mathbb{E}\left((1+c-1)^Z\right)= \sum_{k=0}^n\mathbb{E}{Z \choose k}(c-1)^k,$$ and similar proposition holds for $Y$. It is enough to show that $$ \mathbb{E}{Z \choose k}\leq \mathbb{E}{Y \choose k}, $$ since this inequality implies the lemma. For the binomial distribution, we have

$$\mathbb{E}{Y \choose k}=\sum_{t=0}^n {t \choose k}P(Y=t)= \sum_{t=k}^n  {t \choose k}p^t(1-p)^{n-t}{n \choose t}= {n \choose k}p^k$$ in view of the binomial theorem and the identity  $${n\choose t}{t \choose k}={n\choose k}{n-k \choose t-k}. $$ 
On the other hand,
$$\mathbb{E}{Z \choose k}=\sum_{t=0}^n {t \choose k}P(Z=t)= \sum_{t=k}^n  {t \choose k} \frac{ {pN \choose t}  {(1-p)N \choose n-t}   }{{N \choose n}}=\frac{ {n\choose k} {pN \choose k} }{{N \choose k  }}.$$ Indeed, if we multiply the term $\frac{ {n\choose k} {pN \choose k} }{{N \choose k  }}$ by $$1= \sum_{t=k}^{pN} \frac{{pN-k \choose t-k}{(1-p)N \choose n-t}}{{N-k \choose n-k}},$$ the identity corresponding to the  sum of the probabilities of a hypergeometric distribution $(N-k, pN-k, n-k)$, we end up at the equation $$ \sum_{t=k}^n  {t \choose k} \frac{ {pN \choose t}  {(1-p)N \choose n-t}   }{{N \choose n}} = \frac{ {n\choose k} {pN \choose k} }{{N \choose k  }} \sum_{t=k}^{pN} \frac{{pN-k \choose t-k}{(1-p)N \choose n-t}}{{N-k \choose n-k}}. $$

\noindent By expanding the binomial coefficients, we get that the expressions on the two sides are identical.
Hence,  the claim  $\mathbb{E}(c^Z)\leq  \mathbb{E}(c^Y)$ is equivalent to   the straightforward inequality $$\frac{{pN \choose k}  }{{N \choose k}} \leq p^k \mbox{ \  for } p\in [0,1].$$
\end{proof}

\begin{theorem} \label{proby}
If $p<1$, then $$F_p(n) <\frac{2}{\log_2\frac{2}{1+p}}\log_2{n}+1 {\mbox \ \ holds.}$$
\end{theorem}

\begin{proof}
Let $G_n$ be a random $p$-dense semicomplete digraph, that is, bi-oriented edges of cardinality $p{n \choose 2}$ are chosen randomly and the remaining edges are oriented at random, independently from the others. Let $A_i$ be the event that a given $k$-subset $X^{(k)}_i$ of $V(G_n)$ induces a transitive semicomplete digraph, and $A_i^*$ be the variable that counts distinct transitive subtournaments in $X^{(k)}_i$. Note that $P(\bigcap_i{\overline{A_i}})>0$ implies $F_p(n)<k$.\\
By definition, the event $\overline{A_i}$ is equivalent to the event $\{A_i^*=0\}$, so by  Markov's inequality, we get that $\mathbb{E}(\sum_i{A_i^*})<1$ implies $F_p(n)<k$.\\ We evaluate the expected value by separating all orderings of $X^{(k)}$ for a possible transitive tournament, as follows.

 $$\mathbb{E}(\sum_i{A_i^*})=\sum_i\mathbb{E}({A_i^*})={n \choose k}k!\sum_j{\frac{2^{j}}{2^{k \choose 2}}P(j  \mbox{ bi-oriented edges exist in } X^{(k)}_i)}.$$
 
 Let $Z$ be a  random variable, which counts the bi-oriented edges in the edge set of $X^{(k)}_i$. Clearly, $Z$  has distribution $Z\sim Hypergeom\left({n \choose 2}, p{n \choose 2}, {k \choose 2}\right)$. According to Lemma \ref{moment}, $\mathbb{E}(c^{Z})\leq  \mathbb{E}(c^Y)$, where $Y\sim Binom\left( {k\choose 2}, p\right)$. Applying this in the former equality with  $c=2$, we get
 $$\mathbb{E}(\sum_i{A_i^*})\leq {n \choose k}\frac{k!}{2^{k \choose 2}}\sum_j{{2^{j}}p^j(1-p)^{{k \choose 2}-j}{{k \choose 2}\choose j}}={n \choose k}k!\left(\frac{1+p}{2}\right)^{k \choose 2}.  $$
Hence if $n^k< \left(\frac{2}{1+p}\right)^{k \choose 2} \Leftrightarrow 2k\log_2{n}<k(k-1)\log_2\frac{2}{1+p}$ holds for $k$, then $F_p(n)<k$, that is, $F_p(n) <\frac{2}{\log_2\frac{2}{1+p}}\log_2{n}+1$.
\end{proof}

Theorem \ref{proby} implies that for any fixed density $p<1$, both $f_p(n)$ and $F_p(n)$ is of order $\log n$. The following theorem yields the same conclusion based on a construction.

\begin{theorem}\label{kons}
Let $t\leq n$ be a positive integer, and $p:=1-\frac{1}{t}$. Then $F_p(n)\leq tF_0(\left\lceil \frac{n}{t}\right\rceil)\leq 2t\log_2(\frac{n}{t})$.
\end{theorem}

\begin{proof}

Partition the vertices into $t$ classes of size  $\left\lceil \frac{n}{t}\right\rceil$ or $\left\lfloor \frac{n}{t}\right\rfloor$. Let each class span an oriented graph containing no transitive tournaments of size more than $F_0(\left\lceil \frac{n}{t}\right\rceil)$. Finally, let the edges between the partition classes be bi-oriented. This digraph contains  no transitive tournaments of size more than $tF_0(\left\lceil \frac{n}{t}\right\rceil)$, while a simple calculation shows that the density of the bi-oriented edges is at least $p$.
 The bound thus follows from inequality~(\ref{szeker}).  
\end{proof}

\begin{remark} In view of Observation \ref{triv}, Theorem \ref{kons} can be extended to any $p \in (0,1)$ by applying the theorem for an integer $t$ for which  $p < 1-\frac{1}{t}$ holds.
\end{remark}

\begin{obs}
Theorem \ref{proby} is useful if $p$ is rather small, while the upper bound of Theorem \ref{kons} is better if $p>1/2$. Note that the probablistic bound can be improved using a more involved calculation in Lemma \ref{moment}.
\end{obs}

\subsection{The Erd\H os--Szemer\'edi argument, lower bound on $f_p(n)$ }

In their celebrated paper \cite{ESZ}, Erd\H os and  Szemer\'edi presented a  variant of the Ramsey problem, where a dependence on the density of the graph is taken into consideration.

\begin{theorem}\cite{ESZ} Let $G$ be a graph on $n$ vertices and $\frac{1}{k}\binom{n}{2}$ edges. Then either $G$ or its complement contains a complete graph $K_a$ with $a>C\frac{k}{\log k}\log n$, where $C$ is a constant independent of $k$ and $n$.  
\end{theorem}

Following their ideas and the exact result of \v Culik, the following quantitative form can be confirmed.

\begin{theorem}\label{eszi} Let $G$ be a graph on $n$ vertices and $\frac{1}{k}\binom{n}{2}$ edges  with $k\ge 6$. Then either $G$ or its complement contains a complete graph $K_a$ with $a>\frac{1}{10}\frac{k}{\log k}\log n$.  
\end{theorem}

($\log n$ denotes the natural logarithm)

Observe that this assertion indeed implies the lower bound of Theorem \ref{fo1}, hence 
$$\frac{1}{5(1-p)}\frac{1}{\log{\frac{2}{1-p}}}\log n \leq f_p(n) $$ holds
if $p\ge 2/3$.

\begin{proof}[Proof of Theorem \ref{fo1}, lower bound]
We can assume that the number of edges colored with only red color is at most $\frac{1-p}{2}\binom{n}{2}$. Apply Theorem \ref{eszi} with $k= \frac{2}{1-p}$ where $p>2/3$ (ensuring $k>6$). It yields a suitably large monochromatic clique.
\end{proof}

To prove the quantitative variant of the Erd\H os--Szemer\'edi theorem, first recall the
Zarankiewicz problem about determining the function $\z(\mu,\nu,s,t)$.

\begin{defi} Given positive integers $\mu, \nu, s, t$, let $\z(\mu,\nu,s,t)$ be the maximum number of ones in a $(0,1)$ matrix of size $\mu\times \nu$ that does not contain an all ones submatrix of size $s\times t$. 
\end{defi}

Equivalent definition is the following

\begin{defi} A bipartite graph $G = (A,B; E)$ is $K_{s,t}$-free if it does not contain $s$ nodes
in $A$ and $t$ nodes in $B$ that span a subgraph isomorphic to $K_{s,t}$. The maximum number
of edges a $K_{s,t}$-free bipartite graph of size $(\mu,\nu)$ may have is denoted by $\z(\mu,\nu, s,t)$, and is called a Zarankiewicz number.
\end{defi}

The order of magnitude of the Zarankiewicz numbers in general was obtain by 
K\H ov\'ari, S\'os and Tur\'an, and reads as follows.

\begin{theorem}\cite{KST} 
$\z(\mu,\nu;s,t) \leq (s-1)^{1/t} (\nu-t+1) \mu^{1-1/t} + (t-1)\mu$
\end{theorem}

It was later improved by many authors, see \cite{ Furedi, Niki}.

However, in the special case when one color class is much larger than the other, the value can be determined exactly  in view of an early result of \v Cul\'ik.

\begin{theorem}[\v Cul\'ik \cite{Cc}]\label{Culik} If $1\leq s\leq \mu$ and $ (t-1)\binom{\mu}{s}\leq \nu$, then $$\z(\mu,\nu,s,t)=(s-1)\nu+(t-1)\binom{\mu}{s}.$$
\end{theorem}

\begin{proof}[Proof of Theorem \ref{eszi}]  Let $G$ has $n$ vertices, and  average degree $\overline{d}=\frac{n}{k}$. Then a subgraph $G'\subseteq G$ exists such that $\forall v\in V(G') : d_{G'}(v)\leq \frac{2n}{k}$ and $|V(G')|=n/2=:n'.$

Consider a largest maximal independent set $I$ in $G'$. We may assume  that $\mu=|I|\leq \frac{1}{10} \frac{k}{\log k}\log n$, otherwise  we are done. Let $W:=V(G')\setminus I$ and let $H$ denote the complement of the bipartite graph induced by $I$ and $W$ w.r.t. the complete bipartite graph on the same clusters. Applying  Theorem \ref{Culik} of \v Cul\'ik, we obtain a large complete subgraph $K_{s,t}$ in $H$. \\
Indeed, $E(H)>\mu(n'-\mu-\frac{4n'}{k})$ where $\mu\leq \frac{1}{10} \frac{k}{\log k}\log n$, while setting $s= (1-5/k)\mu\leq  (1-~5/k)~C\frac{k}{10\log k}\log n$ implies

$$\z(\mu,n'-\mu,s,t)= (s-1)(n'-\mu)+(t-1)\binom{\mu}{s}$$  if the conditions of the theorem are verified. To this end, set $t=\sqrt{n}$, and consider the bound 

$$\binom{\mu}{s}= \binom{\mu}{\mu-s}\leq \left(\frac{\mu\cdot e}{\mu-(1-5/k)\mu}\right)^{\mu-(1-5/k)\mu}\leq \left(\frac{k\cdot e}{5}\right)^{\frac{5}{10}\frac{\log n}{\log k}} <\sqrt{n'/4},$$ thus $(t-1)\binom{\mu}{s}< \sqrt{n}\sqrt{n'/4}\leq n'-\mu$.

Easy computation shows that $\z(\mu,n'-\mu,s,\sqrt{n})< s(n'-\mu) < \mu(n'-\mu-\frac{4n'}{k}) <E(H)$ holds, as $$  \frac{s}{\mu} = (1-5/k) < 1-\frac{4n'}{k(n'-\mu)}.$$

Hence in graph $H$, there exists a set $T$ of $\sqrt{n}$ points in $W\subset V(H)$ which are connected to a set $S$ of $s$ points of $I$. $\omega(H)$ cannot exceed  $\mu-s$, as $I$ was a maximal independent set in $G'$. Consequently, bounds on the Ramsey numbers gives lower bound on the independence number of $T$ and the clique number of $G'$ as well.
Namely, using the bound of Erd\H os and Szekeres \cite{esz}, $$R(\mu-s, \mu)< \binom{2\mu-s}{\mu-s} < \left(\frac{(2\mu-s)\cdot e}{\mu-s}\right)^{\mu-s}\leq \left(\frac{(k+5)\cdot e}{5}\right)^{\frac{5}{10}\frac{\log n}{\log k}} < k^{\frac{1}{2}\frac{\log n}{\log k}}= \sqrt{n},  $$ if $k\geq 6$.
This bound holds for every $\mu\leq \frac{1}{10} \frac{k}{\log k}\log n$,
thus $H\mid_T$ must contain an independent set of size at least $\frac{1}{10} \frac{k}{\log k}\log n$, completing the proof.
\end{proof}


\section{The case $m=o(n^2)$}

The cardinality of the bioriented or bicolored edges is assumed to be $(1-o(1)){n \choose 2}$ from now on.\\
 This section is built up as follows.
  First we recall some Tur\'an-type lemmas, which will be key ingredients in the proofs of Theorem \ref{fo2} (i) and (ii). We continue with  the proofs of  the upper and lower bounds of Theorem \ref{fo2}, distinguishing  between the cases $m\leq n$ and $m \geq n$.
  
\subsection{Preliminary lemmas}

\begin{theorem}[Tur\'an theorem, Caro-Wei bound]\label{Turaan}\cite{Caro, Tura, wei} Let $H$ be a simple graph with degree sequence $0<d_1\leq \ldots...\leq d_n$. Then there exists an induced subgraph $H'$ of $H$  on $n'\geq \sum_{i=1}^n \frac{1}{d_i+1}$ nodes, containing no edges. That is,  $$\alpha(H)\geq \sum_{i=1}^n \frac{1}{d_i+1}$$ holds for the independence number of $H$.
\end{theorem} 

\begin{cor}\label{Turaaan}\cite{gr} Let $H$ be a simple graph on $n$ vertices and of $m$ edges. Then  $$\alpha(H)\geq  \frac{n}{2m/n+1}.$$
\end{cor}

The Caro-Wei bound has been generalized in many ways. Caro, Hansberg and Tuza studied \cite{Caro1, Caro2} $d$-independent subsets $S\subseteq V$ of the vertex set $V=V(G)$ of a  graph $G$, which are 
sets of vertices such that the maximum degree in the graph induced by $S$ is at most $d$. Alon, Kahn and Seymour called a graph $H$ $d$-degenerate if every subgraph of it contains a vertex of degree smaller than $d$ and bound the  maximum number $\alpha_d(G)$ of vertices of an induced $d$-degenerate 
subgraph of $G$ in \cite{AKS}. Observing that $\alpha_1(G)=\alpha(G)$ and $2$-degenerate graph are forests, we can formulate their result  - similarly to Theorem \ref{Turaan} - as follows.

\begin{lemma}\label{csiki}\cite{AKS} Let $H$ be a simple graph with degree sequence $0<d_1\leq \ldots...\leq d_n$. Then there exists an induced subgraph $H'$ of $H$  on $$n'\geq \sum_{i=1}^n \frac{2}{d_i+1}$$ nodes, containing no cycles.
\end{lemma}

For the sake of completeness we give a short proof.

\begin{proof} Take an arbitrary arrangement of the vertices of $H$. Label those vertices which have at most one preceding neighbour in the permutation. It immediately follows that labeled vertices cannot span a cycle. If the arrangement is chosen uniformly at random, then every vertex $i$ gets a label with  probability $\frac{2}{d_i+1}$, thus the expected value of the labeled vertices is exactly $\sum_{i=1}^n \frac{2}{d_i+1}$. This implies the existence of a suitable subgraph $H'$. 
\end{proof}

\noindent If $n\leq m$, Lemma \ref{csiki} implies 

\begin{cor}\label{aloncikk}\cite{AKS} Let $H$ be a simple graph on $n$ vertices and of $n\leq m$ edges. Then  $$\alpha_2(H)\geq  \frac{2n}{2m/n+1}.$$
\end{cor}

\subsection{Subcase $m\leq n$}

\noindent We start with the case $m\leq n$  to demonstrate the proof techniques of Theorem \ref{fo2}, underline the  difference between the functions $f(n,m)$ and $F(n,m)$ in consideration, furthermore to determine explicitly its values via Theorem \ref{fo2} (i).

\begin{proof}[Proof of Theorem \ref{fo2} (i)]
First we prove $f(n,m)= n- \left\lfloor\frac{m}{2}\right\rfloor$. Consider those edges which are colored with precisely one color. Assume that the color red is assigned to at least as many edges as the color blue. Thus there are at most $ \left\lfloor \frac{m}{2}\right\rfloor$ blue edges. Delete a minimal covering (vertex) set of the blue edges. The remaining vertex set induces a graph where the color red is assigned to every edge. On the other hand, if both  the blue and the red edge set is independent, equality holds.

Next, we prove that  $F(n,m)\leq n- \left\lfloor \frac{m}{3}\right\rfloor$ by a construction. Take $ \left\lfloor \frac{m}{3}\right\rfloor$ disjoint triangles, and orient round the edges in each triangle. Let all the edges between different triangles  be bi-oriented edges. Clearly, at most is sharp. Indeed, let $H$ be the simple graph constructed from $G$ by deleting the bi-oriented edges, and omitting the orientations of the remaining oriented edges and the isolated vertices. Let $0<d_1\leq \ldots...\leq d_s$ be the degree sequence of $H$.  We can obtain a subgraph $H'$ on at least $\sum_{i=1}^s \frac{2}{d_i+1}$ nodes which do not span cycles, according to Lemma \ref{csiki}. Consider the directed subgraph of $G$ restricted to the vertex set corresponding to the vertex set of $H'$. Since $H'$ was a forest, the graph obtained from $G$ by deleting the bi-oriented edges and restricted to the vertex set of $H'$ has a topological ordering. This provides a transitive subtournament of $G$ on  $\left\lceil\sum_{i=1}^s \frac{2}{d_i+1}\right\rceil$ nodes. Hence Lemma \ref{csoki} completes the proof.
\end{proof}

\begin{lemma}\label{csoki}
Let $G$ be a simple graph with degree sequence $0<d_1\leq \ldots\leq d_s$. 

If $G$ has $m$ edges and $s\leq n$, then
$$\sum_{i=1}^s \frac{2}{d_i+1}+n-s \geq n- \frac{m}{3}.$$

\end{lemma}

\begin{proof}

Observe that the following inequality holds for every positive integer $x$: $$ \frac{2}{x+1} -1\geq -\frac{x}{6}. $$   Since  $\sum_{i=1}^s d_i=2m$, summing it for every degree of $G$ confirms the statement.
\end{proof}

\subsection{Subcase $n\leq m$}

Next we show that 
$n\leq m$ implies that (A)
 $f(n,m)\geq  \frac{n}{m/n+1}$
 and 
 (B) $F(n,m)\geq \frac{2n}{2m/n+1}$ hold.

\begin{proof}[Proof of Theorem \ref{fo2}(ii)]
Part (A). Consider the edges which have exactly one color, and suppose that the color red is assigned to at least as many edges as the color blue. Thus the cardinality of the blue edges is at most $m/2$.  Let us take  a maximal independent set on the graph of the blue edges. This provides a monochromatic (red) subgraph in the original $K_n$. The cardinality of the maximal independent set is at least $\frac{n}{m/n+1}$ due to Corollary \ref{Turaaan}.

Part (B). Apply Corollary \ref{aloncikk} to the simple graph $H$  obtained from our digraph $G$ by deleting all bi-oriented edges and omitting the orientation for the rest of the edges. Consider the obtained forest subgraph $H'$ of $H$, and the original orientations of the edges of $H'$. These edges in $G$ obviously provides an acyclic orientation, with a suitable topological ordering, hence  the vertex set of $H'$ guarantees a transitive subtournament on at least $\frac{2n}{2m/n+1}$ vertices. \end{proof}



We highlight the connection between the feedback vertex set problems (deadlock recovery) and Theorem \ref{fo2} (ii)(B).
   A feedback vertex set of a graph (or digraph) is a set of vertices whose removal leaves a graph without  (directed) cycles. In other words, each feedback vertex set contains at least one vertex of any cycle in the graph. Since the feedback vertex sets play a prominent role in the study of deadlock recovery in operating systems, it has been studied extensively \cite{Fomin, Razgon, Posa, Festa,  Karp}.   To find the minimal feedback vertex set is NP-complete in the undirected and directed case as well \cite{Karp}.

 Essentially,  $F(n,m)$ is equal to the size of the maximal size of an acyclic subgraph which is guaranteed  to be in a digraph of $m$ directed edges on $n$ vertices, that is, $n-F(n,m)$ is the minimal size of a feedback vertex set, such that the removal of an appropriate set of $n-F(n,m)$ vertices makes any digraph of $m$ edges acyclic. Instead, in the proof of Theorem \ref{fo2} (ii)(B)  we bound the size of the minimal size of a feedback vertex sets of undirected graphs of $m$ edges, which generally provides a fairly weaker bound, even if they both lead to exact results when $m$ is small via Theorem \ref{fo2} (i). 

\bigskip

\noindent Before stating upper bounds for $f(n,m)$ and prove Theorem \ref{fo2} (iii), recall that $\vec{R}(n)$ denoted the directed Ramsey number, and $F(\vec{R}(k+1)-1)=k$. 

\begin{prop}\label{bound}
Let $|E_{bi}(G)|={n \choose 2}-m$, where $m=  n\cdot( \frac{\vec{R}(k+1)}{2}-1)$. Suppose that ${\vec{R}(k+1)-1}~\mid~n$. Then 
$$F(n,m)\leq  \frac{k}{\vec{R}(k+1)-1}\cdot n .$$


\end{prop}

\begin{proof} There exists a tournament $T$ on ${\vec{R}(k+1)-1}$ vertices which does not contain a transitive tournament of size $k+1$. Take $\frac{n}{\vec{R}(k+1)-1}$ disjoint copy of it, and join every pair of vertices with bi-oriented edges which are not in the same copy of $T$. Clearly, the size of the maximal transitive tournament of this graph is at most $k \cdot \frac{n}{\vec{R}(k+1)-1}$.
\end{proof}

\begin{prop}\label{bound2}
Let $|E_{RB}(G)|={n \choose 2}-m$, and let $C$ denote  $\lfloor m/n \rfloor$, and suppose $C+1  \leq  n/(C+~1)\in ~\mathbb{Z}^+$. 
Then 
$$f(n,m)\leq  \frac{n}{C+1} .$$


\end{prop}

\begin{proof} Let $r$ be the remainder in the division $n:(C+1)$.
Take the disjoint union of $\lfloor n/(C+1)\rfloor$  cliques of size $C+1$, and a clique of size $r$ and call it $K$. Observe first that the number of edges is at most $m/2$ in $K$. Consider now $K'$ and $K''$, two copies of $K$. Since $C+1  \leq  n/(C+~1)$, it is easy to see that one can pack them into a complete graph $K_n$, i.e. there exists an edge-disjoint
placement of them onto the same set of $n$ vertices. In fact, this is a very special case of the famous Hajnal-Szemer\'edi theorem \cite{eha}, which states that if $H$ is an $n$-vertex graph with $\Delta(G) \leq q$, then $H$ packs with the graph H' whose components are complete graphs of size $\lfloor n/(q+1)\rfloor$ or $\lceil n/(q+1)\rceil$. 
Color the edges of $K'$ blue and the edges of $K''$ red.
Hence, if all the other edges of $K_n$ are colored with both colors, then at most $m$ will be the number of  unicolored edges, while a monochromatic clique has at most $\lfloor n/(C+1)\rfloor+1$ vertices due to the pigeon-hall principle.
\end{proof}

As a consequence Theorem \ref{fo2} (iii) follows.

\noindent Theorem \ref{fo2} (ii) and \ref{bound2} implies together that equality can be attained in the two-colored case, if some divisibility conditions hold and  $m$ is not too large.

\begin{cor} If  $m/n=C \in \mathbb{Z}^+$, and $ C+1 \leq n/ (C+1)\in \mathbb{Z}^+$, then 
$$f(n,m)= \frac{n}{m/n+1}.$$ This method thus provides exact result till $m \leq n^{3/2}$ holds. 
\end{cor}
If the order of magnitude is bigger than $n^{3/2}$, the graph cannot contain a disjoint union of packings  of monochromatic cliques of size $C+1$, from  both red and blue color, hence the lower bound cannot be attained.

\section{Improving the upper bound for $F(n)$}

In the following section, we consider simple oriented graphs. 

\begin{lemma} [Lov\'asz Local Lemma] \cite{local}
Let $A_1, A_2,\ldots, A_m$ be a series of events such that each event occurs with probability at most $P$ and such that each event is independent of all the other events except for at most $d$ of them. If $Pe(d+1)<1$, then there is a nonzero probability that none of the events occurs.
\end{lemma}

\begin{theorem}
$F(n)<2\log_2{n}-1 +o(1)$. 
\end{theorem}

\begin{proof}
Fix $n\geq55$, and let $\sqrt{n}>k>4$.
Let $T_n$ be a random tournament such that every edge is oriented in one direction uniformly at random, independently from the orientations of all the other edges. Let $A_i$ be the event that a given $k$-subset $X^{(k)}_i$ of $V(T_n)$ is a transitive tournament.\\ Clearly, $P(A_i)=\frac{k!}{2^{k\choose 2}}$.\\ $A_i$ is independent of all events $A_j$ for which $|X^{(k)}_i \cap X^{(k)}_j|\leq 1$. This yields
$$ d\leq {n \choose k} - {n-k \choose \ k} - k{n-k \choose k-1} = \sum_{j=2}^k{{k \choose j}{n-k \choose k-j}}.$$

Increasing the right hand side, we get
$$ d< {k \choose 2}{n-k \choose k-2}\left( 1+ \left[\frac{k-2}{3}\frac{k-2}{(n-2k+3)}\right]+ \left[\frac{k-2}{3}\frac{k-2}{(n-2k+3)}\right]^2+ \left[\frac{k-2}{3}\frac{k-2}{(n-2k+3)}\right]^3+ \ldots \right)$$
Hence
$$d< {k \choose 2}{n-k \choose k-2}\left( \frac{1}{1-\frac{(k-2)^2}{3(n-2k+3)}}\right).$$
Applying $\sqrt{n}>k>4$, we get that the sum of the above geometric progression is less than $1.2$.

In view of the Lov\'asz Local Lemma, if $e(d+1)P(A_i)<1$, then\\ $P$(no transitive k-subtournaments in $T_n$) $>0$, thus $f(n)<k$.\\
Therefore if $$e\cdot \frac{k^2}{2}\frac{n^{k-2}}{(k-2)!}1.2\cdot \frac{k^2(k-2)!}{2^{k\choose 2}}<1,$$ then $F(n)<k$.
This implies $$(k-2)\log_2{n}<{k\choose 2}-4\log_2{k}-0.7.$$ Consequently,
 $$2\log_2{n}-1<k+\frac{2-8\log_2{k}-1.4}{(k-2)},$$
which completes the proof.
\end{proof}

\noindent \textbf{Acknowledgments}\\
I am very grateful to \'Agnes T\'oth, Andr\'as Gy\'arf\'as and Yair Caro for their valuable suggestions 
which helped  to substantially improve the manuscript.



\begin{thebibliography}{99}
\small{
\bibitem{AKS} {\sc N. Alon, J. Kahn, P. D. Seymour}, Large Induced Degenerate Subgraphs, {\em Graphs and Combinatorics} \textbf{3},  (1987) 203 - 211.

\bibitem{alon} {\sc N. Alon, J. H. Spencer}, The probabilistic method, {\em New York: Wiley-Interscience} {\bf} (2000).

\bibitem{balogh} {\sc  J. Balogh, P. Hu, M. Simonovits,}  Phase transitions in Ramsey-Tur\'an theory, \emph{J. Combinatorial Theory B} \textbf{114},  (2015) 148-169.

\bibitem{bang} {\sc J. Bang-Jensen, G. Gutin}, Digraphs: Theory, Algorithms and Applications, {\em Springer} (2009).



   
\bibitem{Caro} {\sc  Y. Caro}, New results on the independence number,  {\em Tech. Report}, Tel-Aviv University, (1979).
   
\bibitem{Caro1} {\sc  Y. Caro, A. Hansberg}, New approach to the k-independence number of a graph, {\em Elect. J. Comb.} \textbf{20}(1), (2013), P33. 
   
\bibitem{Caro2}  {\sc  Y. Caro, Z. Tuza}, Improved lower bounds on k-independence, {\em J. Graph Theory} \textbf{15}(1), (1991) 99 - 107.
   
\bibitem{Fomin}  {\sc J. Chen, F. V. Fomin, Yang Liu, S. Lu, Y. Villanger},  Improved algorithms for feedback vertex set problems, {\em J. Comp.  Sys. Sciences} \textbf{74}(7), (2008) 1188 - 1198.

\bibitem{Razgon}  {\sc  J. Chen, Y. Liu, S. Lu, B. O'Sullivan, I. Razgon},  A fixed-parameter algorithm for the directed feedback vertex set problem, {\em Journal of the ACM} \textbf{55} (5), Art. 21, (2008).

\bibitem {conl} {\sc D. Conlon}, A new upper bound for diagonal Ramsey numbers,  {\em Ann. Of Math.} {\bf 170}(2), (2009) 941 - 960.

\bibitem{Cc} {\sc K. \v Cul\'ik}, Teilweise L\"osung eines verallgemeinerten Problem von K. Zarankiewicz. \emph{Ann. Soc. Polon.
Math.} \textbf{3}, (1956) 165-168.

\bibitem {erd} {\sc P. Erd\H os, } Some remarks on the theory of graphs,  {\em Bull. Amer. Math. Soc.} {\bf 53}(4), (1947) 292 - 294.
 

 
\bibitem{erdos} {\sc P. Erd\H os, L. Moser}, On the representation of directed graphs as unions of orderings, {\em Publ. Math. Inst. Hungar. Acad. Sci.} {\bf 9}, (1964) 125 - 132.
 
\bibitem{ER} {\sc P. Erd\H os, R. Rado}, Partition relations and transitivity domains of binary relations, {\em  J. London Math. Soc.} {\bf 1}(1), (1967) 624 - 633. 
  
 \bibitem{esz} {\sc P. Erd\H os, G. Szekeres}, A combinatorial problem in geometry {\em Compositio Mathematica}, {\bf 2}, (1935) 463 - 470. 
 
\bibitem{ESZ} {\sc P. Erd\H os, E. Szemer\'edi}, On a Ramsey type problem, \emph{Periodica Math. Hung.} \textbf{2}(1-4), (1972)  295-299.




\bibitem{local} {\sc P. Erd\H os, L. Lov\'asz}, Problems and results on 3-chromatic hypergraphs and some related questions, {\em Colloq. Math. Soc. J. Bolyai 11, North Holland, Amsterdam}, (1975) 609 - 627. {\bf} 

\bibitem{Posa} {\sc P. Erd\H os, L.  P\'osa},  On independent circuits contained in a graph, {\em Canadian J. Math.} \textbf{17} (1965) 347 - 352.
  
\bibitem{Festa} {\sc  P. Festa, P.M. Pardalos, M.G.C. Resende}, Feedback set problems. {\em In: Handbook of Combinatorial
Optimization, Supplement} {A},  Kluwer Academic, Dordrecht (1999) pp. 209 - 258.

\bibitem{Furedi}  Z. F\"uredi, An upper bound on Zarankiewicz's problem, \emph{Comb. Prob. Comput.} \textbf{5.} (1996) 29-33.

\bibitem{gr} {\sc J.R. Griggs}, Lower bounds on the independence number in terms of the degrees, {\em J. Combin. Theory,
Ser. B}, {\bf 34}, (1983) 22 - 39.

\bibitem{eha} {\sc A. Hajnal, E. Szemer\'edi}, Proof of a conjecture of Erd\H os, {\em  Combinatorial
Theory and its Applications } Vol. II,
North-Holland, New York-London (1970)  601 - 603.

\bibitem{Karp} {\sc R.M. Karp}, Reducibility among combinatorial problems {\em  In: Complexity of Computer Computations},
 Plenum, New York (1972) pp. 85 - 103.
 
\bibitem{KST}{\sc  T. K\H ov\'ari,  V. T. S\'os, P.  Tur\'an}, On a problem of K. Zarankiewicz, \emph{Colloquium Math.} \textbf{3}, (1954) 50-57.  
 

\bibitem{Lars} {\sc J. A. Larson, W. J. Mitchell}, On a problem of Erd\H os and Rado, {\em Annals of Combinatorics} {\bf 1}(1), (1997) 245 - 252.



\bibitem{lara} {\sc V. Neumann-Lara}, A Short Proof of a Theorem of Reid and Parker on
Tournaments , {\em Graphs and Combinatorics}, {\bf 10}, (1994) 363 - 366.

\bibitem{Niki} V. Nikiforov,  A contribution to the Zarankiewicz problem, \emph{Lin. Alg. Appl.}
\textbf{ 432(6)},  (2010) 1405-1411.

\bibitem{reid} {\sc K. B. Reid, E. T. Parker}, Disproof of a conjecture of Erd\H os and Moser on tournaments, {\em J. Combin. Theory},  {\bf 9}, (1970) 225 - 238.

\bibitem{flores} {\sc A. Sanchez-Flores}, On tournaments free of large transitive subtournaments, {\em Graphs and Combinatorics} {\bf 14},(1998) 181 - 200.



\bibitem{sos} {\sc M. Simonovits,   V. T. S\'os,} Ramsey-Tur\'an theory, \emph{Discrete Mathematics}, \textbf{229}(1), (2001) 293-340.

\bibitem{spen} {\sc J. Spencer}, Ramsey's theorem - a new lower bound, {\em J. Combin. Theory Ser. A} {\bf 18}, (1975) 108 - 115. 

\bibitem{stearn} {\sc R. Stearns}, The voting problem,  {\em Amer. Math. Monthly } {\bf 66}, (1959) 761 - 763.

\bibitem{tabib} {\sc C. Tabib}, About the inequalities of Erd\H os and Moser on the largest transitive subtournament
of a tournament,  {\em Combinatoire \'enum\'erative, Lecture Notes in Math. 1234}, {\bf} Berlin Heidelberg New York: Springer (1986) 308 - 320. 


\bibitem{Tura} {\sc P. Tur\'an}, An extremal problem in graph theory (hungarian), {\em Mat Fiz Lapok}  \textbf{48}, (1941) 436 - 452.


\bibitem{wei}  {\sc V.K. Wei}  A lower bound on the stability number of a simple graph, {\em Tech. Memorandum} No. 81-
11217-9. Bell Laboratories, 1981.}

\end{thebibliography}
\end{document}